\documentclass[12pt,a4paper]{article}
\usepackage[scale=0.85]{geometry}
\usepackage[T1]{fontenc}
\usepackage{lmodern}
\usepackage{amsmath,amssymb,amsthm}
\usepackage{microtype}
\usepackage{enumitem}
\usepackage{graphicx}
\usepackage{booktabs}
\usepackage{needspace}
\usepackage[colorlinks]{hyperref}
\hypersetup{pdftitle={Cubic vertices in minimal braces}}
\allowdisplaybreaks[1]

\newtheorem{theorem}{Theorem}[section]
\newtheorem{lemma}[theorem]{Lemma}
\newtheorem{proposition}[theorem]{Proposition}
\newtheorem{conjecture}[theorem]{Conjecture}
\newtheorem{question}[theorem]{Question}

\title{Cubic vertices in minimal braces}
\author{{\small\bf Yipei Zhang$^{1}$, Xiumei Wang$^{2}$}\thanks{Corresponding author. Email address: wangxiumei@zzu.edu.cn}\\
{\small $^{1}$School of Mathematics and Statistics, North China University of Water Resources and Electric Power,}\\
{\small Zhengzhou, Henan 450046,  China}\\
{\small $^{2}$School of Mathematics and Statistics, Zhengzhou University,}\\
{\small Zhengzhou, Henan 450001,  China}}
\date{}

\makeatother
\begin{document}
\maketitle
\begin{abstract}
Braces play a fundamental role in matching theory, as they, together with bricks, constitute the basic building blocks of matching covered graphs in the tight cut decomposition.
A brace is minimal if deleting any edge from it results in a graph that is not a brace.
We prove that every perfect matching of a minimal brace of order at
least six contains an edge whose ends both have degree three.
Consequently, every such brace has at least three edges with this
property. Combining this result with the forest structure induced
by the noncubic vertices, we show that every minimal brace $G$
of order $n\ge6$ and size $m$, other than $K_{3,3}$, satisfies
$n_3(G)\ge\max\left\{8,\left\lceil\frac{2n+8}{5}\right\rceil,
\left\lceil\frac{m-n+4}{2}\right\rceil\right\}$, where $n_3(G)$ is the number of cubic vertices of $G$.
Finally, we characterize the graphs attaining the constant lower
bound: a minimal brace has exactly eight cubic vertices if and only
if it is isomorphic to $B_8$, $B_{10}$, $Q_{10}^{+}$, or $Q_{12}$.
\end{abstract}

\textbf{Keywords:} minimal brace; cubic vertex; perfect matching

\section{Introduction}

All graphs considered in this paper are finite and loopless,  and parallel edges are allowed unless
the graph is specified to be simple.
For any undefined notation and terminology, we follow \cite{BM08}.
Let $G$ be a graph with the vertex set $V(G)$ and the edge set $E(G)$.
The \emph{order} and \emph{size} of $G$ are defined as  $|V(G)|$ and $|E(G)|$, respectively.
A \emph{matching} of $G$ is a set of pairwise nonadjacent edges of $G$; it is \emph{perfect} if it covers every vertex of $G$.
A connected nontrivial graph $G$ is \emph{matching covered} if  every edge of $G$ belongs to a perfect matching of $G$.

For a nonempty proper subset $X$ of $V(G)$, we use $\partial_G(X)$ to denote the set of all the edges of $G$ with one end in $X$ and the other end in $V(G)\setminus X$, and write $\partial_G(v)$ for $\partial_G(\{v\})$. We shall omit the subscript $G$ when the graph is understood.
We refer to $\partial(X)$ as a {\it cut} of $G$.
A cut $\partial(X)$ of $G$ is {\it trivial} if one of $X$ and $V(G)\setminus X$ has exactly one vertex, and is \emph{nontrivial} otherwise.
A cut $\partial(X)$ of $G$ is {\it tight} if each perfect matching of $G$ contains precisely one edge in $\partial(X)$.
A matching covered graph with no nontrivial tight cut is called a \emph{brace} if it is bipartite, and a \emph{brick} otherwise.
Braces play a fundamental role in matching theory. Together with bricks, they form the basic building blocks of matching covered graphs under tight cut decomposition.

A brace $G$ is \emph{minimal} if $G-e$ is not a brace for every edge $e\in E(G)$. Note that every minimal brace is simple. Minimal braces form a natural extremal subclass of braces, and their edge-minimality imposes strong structural restrictions. A vertex is \emph{cubic} if it has degree three, and \emph{noncubic} otherwise. Since every brace of order at least $6$ has minimum degree at least three, cubic vertices are precisely the vertices of minimum degree. It is therefore natural to ask how many cubic vertices a minimal brace must contain.

An edge whose two ends are cubic is called a \emph{$3$-line}. He et al. \cite{HLZ2026+} proposed the following conjecture.

\begin{conjecture}[\cite{HLZ2026+}]\label{conj:HLZ2026+}
Every minimal brace with at least six vertices has a $3$-line.
\end{conjecture}

In this paper, we prove a strengthening of Conjecture~\ref{conj:HLZ2026+}.

\begin{theorem}\label{thm:transversal}
Let $G$ be a minimal brace of order at least $6$. Then every perfect matching of $G$ contains a $3$-line.
Consequently, $G$ has at least tree $3$-lines.
\end{theorem}

We denote by $n_3(G)$ the number of cubic vertices of a graph $G$, and by $\lceil x\rceil$ the least integer greater than or equal to $x$.
Lou~\cite{Lou} proved that every minimal brace $G$ of order $n\ge6$ has  $n_3(G)\ge \max\{6,\frac{2n+2}{5}\}$;
Mallik et al. \cite{MDK} proved that $n_3(G)\ge\frac{m-n+1}{2}$, where $m$ is the size of $G$.
Using Theorem \ref{thm:transversal}, we obtain Theorem \ref{thm:main}, which strengthens
these bounds.

\begin{figure}[htb]
\centering
\includegraphics[width=0.5\textwidth]{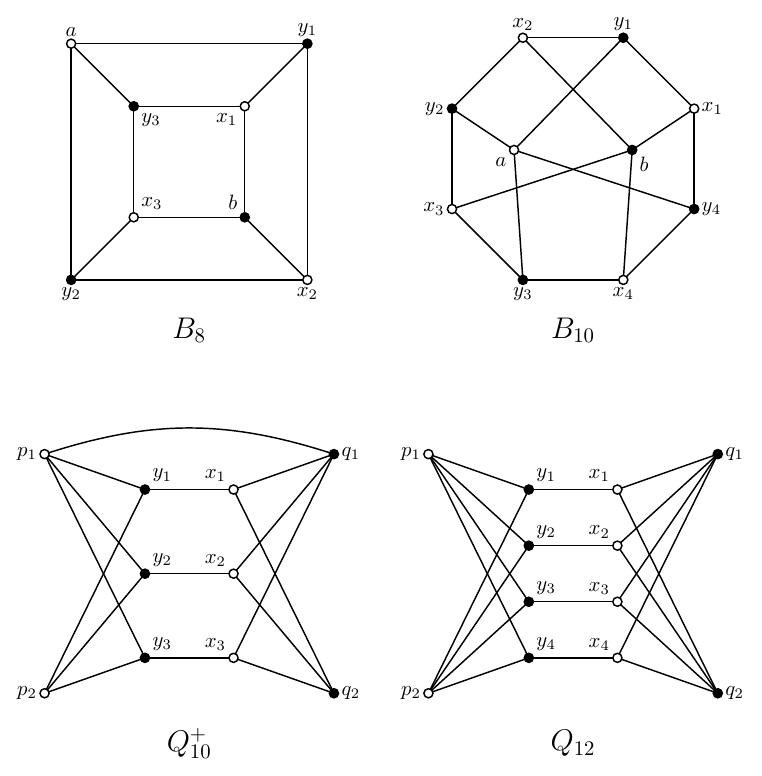}
\caption{The four minimal braces with exactly eight cubic vertices.}
\label{fig:eight-cubic}
\end{figure}

\begin{theorem}\label{thm:main}
Let $G$ be a minimal brace of order $n\ge6$ and size $m$.
Then the following statements hold.
\begin{enumerate}[label=(\roman*),nosep]
\item\label{item:count}
If $G\not\cong K_{3,3}$, then $ n_3(G)\ge
 \max\left\{
 8,\left\lceil\frac{2n+8}{5}\right\rceil,
 \left\lceil\frac{m-n+4}{2}\right\rceil
 \right\}$.
Consequently, $n_3(G)=6$ if and only if $G\cong K_{3,3}$, and no minimal brace has exactly seven cubic vertices.
\item\label{item:eight}
The graph $G$ has exactly eight cubic vertices if and only if it is
isomorphic to one of $B_8, B_{10}, Q_{10}^{+}$ and $Q_{12}$ (as shown in Figure \ref{fig:eight-cubic}).
\end{enumerate}
\end{theorem}

The rest of the paper is organized as follows. In Section \ref{sec:prelim}, we present some notation,  terminology and lemmas. In Sections \ref{sec:transversal} and \ref{sec:main-proof}, we give proofs of Theorems \ref{thm:transversal} and \ref{thm:main}, respectively.

\section{Preliminaries}\label{sec:prelim}

We begin with some notation and terminology.
For a bipartite graph $H$ with colour classes $A_H$ and $B_H$, we write $H=H[A_H,B_H]$, and simply use $A,B$ when there is no ambiguity.
For $X\subseteq V(H)$, let $H[X]$ denote the
subgraph induced by $X$. For disjoint subsets $X,Y\subseteq V(H)$,
let $E_H(X,Y)$ denote the set of all the edges with one end in $X$ and the
other in $Y$. Write $d_H(v)$ for the degree of $v$, and
$\delta(H)$ and $\Delta(H)$ for the minimum and maximum degrees of $H$.
Let $N_H(X)=\bigcup_{v\in X}N_H(v)$.
Subscripts are omitted only when the underlying graph is clear from context.
A bipartite graph is \emph{balanced} if its two colour classes have equal
size.

A \emph{leaf} of a tree is a vertex that has degree exactly one. A \emph{star} is a tree whose edges all meet a single vertex, its center. A \emph{double
star} is obtained by joining the centers of two vertex disjoint stars.  For a graph $H$, $rH$ denotes the disjoint union of $r$ copies of $H$. A subset of $V(H)$ is \emph{stable} if no two of its vertices are adjacent in $H$.
We write $K_2$ for a single edge. For $\ell\ge3$, define the graph $Q_{2\ell+4}$ on colour classes $\{p_1,p_2,x_1,\ldots,x_\ell\}$ and $\{q_1,q_2,y_1,\ldots,y_\ell\}$
by adding the edges $p_jy_i,x_iq_j,x_iy_i$, where $j=1,2$ and
$1\le i\le\ell$.
The graph $Q_{10}^{+}$ shown in Figure~\ref{fig:eight-cubic} is obtained from $Q_{10}$ by adding the edge $p_1q_1$.

A \emph{prism} consists of two vertex disjoint even cycles
$x_1x_2\cdots x_rx_1$ and $y_1y_2\cdots y_ry_1$, where
$r\ge4$, together with the edges $x_iy_i$, $1\le i\le r$.
A \emph{M\"obius ladder} $M_{2h}$ is the cycle $v_1v_2\cdots v_{2h}v_1$ together with the edges
$v_iv_{i+h}$, $1\le i\le h$, where $h\ge3$ is odd.
Note that all prisms and  M\"obius ladders are cubic.
For $h\ge4$, a \emph{biwheel} $B_{2h}$ is the graph obtained from a cycle of order $2h-2$, with colour classes $A,B$,
by adding two nonadjacent vertices $a,b$, joining $a$ to all vertices of
$B$, and joining $b$ to all vertices of $A$. The cycle is called the
\emph{rim}, and $a,b$  the \emph{hubs}.
The cube $B_8$, shown in Figure \ref{fig:eight-cubic}, is the smallest prism as well as the smallest biwheel.
The complete graph $K_{3,3}$ is the smallest M\"obius ladder.
Let $\mathcal G$ be the set that consists of $K_2,C_4$, all
prisms, all M\"obius ladders and all biwheels.

A connected graph $G$ is \emph{$k$-extendable} if it has a matching of size $k$ and every such matching can be extended to a perfect matching of $G$.
Note that $1$-extendable graphs are exactly the matching covered graphs. 
A $k$-extendable graph is \emph{minimal} if deleting any edge results in a graph that is not $k$-extendable.

Let $G$ be a  graph, and $u$  a vertex of degree two that has two distinct neighbors, say $v_1$ and $v_2$. The \emph{bicontraction} of $u$ is the operation of contracting the two edges $uv_1$ and $uv_2$ incident with $u$.
Note that the bicontraction of $u$ in  $G$ may not produce a simple graph.
The retract of $G$, denoted $\widehat{G}$, is the  graph obtained by bicontracting all its vertices of degree two that have two distinct neighbors.
We identify each edge in the retract with the corresponding original edge.

Let $G$ be a brace of order six or more.  An edge $e$ of $G$ is \emph{thin} if $\widehat{G-e}$ is also a brace.
Note that  the brace $\widehat{G-e}$ may not be simple.
A thin edge $e$ of a simple brace $G$ is \emph{strictly thin} if the brace $\widehat{G-e}$ is also simple.
Its \emph{index}, denoted index($e$), is the number of vertices of degree two in $G-e$.
Clearly, index($e$)$\in\{0,1,2\}$.
In $G-e$, a vertex of degree two is referred to as an \emph{inner vertex}, and each of its neighbors is referred to as an \emph{outer vertex}.
For a minimal brace $G$, a pair $(e,F)$ is \emph{minimality-preserving} if $e$
is a strictly thin edge of $G$ and $F$ is a subset of $E(\widehat{G-e})$ such that $\widehat{G-e}-F$ is a minimal brace.

A stable-extension of a bipartite graph $J[A,B]$ is the graph
obtained from $J$ by adding two new vertices $x,y$
and five new edges $xb_1,xb_2,ya_1,ya_2,xy$, where
$\{a_1,a_2,b_1,b_2\}$ is a stable set of $J$ with
$a_1,a_2\in A$ and $b_1,b_2\in B$.
Its colour classes are $A\cup\{x\}$ and $B\cup\{y\}$.
The four  vertices $a_1,a_2,b_1, b_2$ are called the \emph{attachment vertices} of $J$.

For a graph $G$, let $V_k(G)$ denote the set of all the vertices in $G$ of degree
 $k$, and let $V_{\ge k}(G)$ denote the set of all the vertices in $G$ of
degree at least $k$. We omit $G$ when there is no ambiguity.
Thus, $V_3$ is the set of all the cubic vertices of $G$, and so $n_3(G)=|V_3(G)|$.

\begin{lemma}[\cite{Hall1935}]\label{lem:hall}
A bipartite graph $H[A,B]$ has a matching which covers every vertex in  $A$ if and
only if $|N(X)|\ge|X|$ for every $X\subseteq A$.
\end{lemma}

For a balanced bipartite graph $H[A,B]$, Lemma~\ref{lem:hall} implies
that $H$ has a perfect matching if and only if $|N(X)|\ge|X|$
for every $X\subseteq A$.

\begin{lemma}[\cite{FKC}]\label{lem:brace}
Let $H[A,B]$ be a connected bipartite graph with $|A|=|B|\ge3$. The following are equivalent:
\begin{enumerate}[label=(\roman*),nosep]
\item $H$ is a brace.
\item\label{item:brace-neighbourhood} For every nonempty subset $X$ of $A$ such that $|X|\le |A|-2$, we have
\begin{equation}\label{eq:hall}
 |N(X)|\ge |X|+2.
\end{equation}

\item For any distinct vertices $a_1,a_2\in A$ and $b_1,b_2\in B$,
the graph $H-\{a_1,a_2,b_1,b_2\}$ has a perfect matching.
\item\label{item:brace-extendable} $H$ is $2$-extendable.
\end{enumerate}
\end{lemma}

\begin{lemma}[\cite{FKC}]\label{lem:brace-3-connected}
Every brace of order at least six is $3$-connected.
\end{lemma}

By Lemma~\ref{lem:brace}\ref{item:brace-extendable}, a bipartite graph of order at least six
is a minimal brace if and only if it is minimal $2$-extendable.
For a brace $G$ of order at least six, Lemma \ref{lem:brace-3-connected} implies that $\delta(G)\ge 3$.
Thus, $V_{\geq4}(G)$ is the set of noncubic vertices in $G$.

\begin{lemma}[\cite{Lou}]\label{lem:forest}
In a minimal brace of order at least six, the subgraph induced by
the noncubic vertices is a forest.
\end{lemma}

The following result follows from Propositions 2.2, 2.4 and 2.6 in \cite{FKC}.

\begin{lemma}[\cite{FKC}]\label{lem:small}
The following statements hold.
\begin{enumerate}[label=(\roman*),nosep]
\item\label{item:small-six-eight}
The only minimal brace of order $6$ is $K_{3,3}$, and the only one of order $8$ is $B_8$.
\item\label{item:small-ten}
The only minimal braces of order $10$ are $M_{10},B_{10}$ and $Q_{10}^{+}$,
whose sizes are $15$, $16$, and $16$, respectively, and whose numbers of cubic vertices are $10$, $8$, and $8$, respectively.
\item\label{item:small-twelve} The only minimal braces of order $12$ and size at least
$20$ are $B_{12}$ and $Q_{12}$, both of which have size $20$ and have $10$ and $8$ cubic vertices, respectively.
\end{enumerate}
\hspace{2mm}In particular, each of $B_8, B_{10}, Q_{10}^{+}$ and $Q_{12}$ is a minimal brace.
\end{lemma}

\begin{lemma}[\cite{FKC}]\label{lem:size}
Let $G$ be a minimal brace of order $n\ge12$. Then $|E(G)|\le\frac{5n}{2}-10$, and
equality holds if and only if $G$ is isomorphic to either $B_{12}$ or $Q_n$.
In particular, every minimal brace of order $14$ has size at most $25$, and equality holds precisely for $Q_{14}$.
\end{lemma}

The following result follows from Proposition 1.13 and Theorem 3.12 in \cite{FKC}.

\begin{lemma}[\cite{FKC}]\label{lem:induction}
Every minimal brace $G\notin\mathcal G$ has a minimality-preserving
pair $(e,F)$ of index $k\in\{1,2\}$. Let
$J=\widehat{G-e}-F$. Then
\[
 |V(J)|=|V(G)|-2k, \
 |E(J)|=|E(G)|-1-2k-|F|,\ |F|\le k+1.
\]
Moreover, either $F=\emptyset$, or $F$ satisfies the following properties in $G-e$:
\begin{enumerate}[label=(\roman*),nosep]
\item\label{item:induction-one} If $k=1$, all edges of $F$ are incident with one outer vertex in $G-e$;
\item\label{item:induction-one-1} If $k=1$ and $|F|=2$, then $G$ is isomorphic to a stable-extension of $J$;
\item\label{item:induction-two} If $k=2$, there are two adjacent outer vertices $u,w$ such that
$F\subseteq\partial_G(u)\cup\partial_G(w)$;
\item\label{item:induction-ends} For every edge $f\in F$, an end of $f$ is cubic  if and
only if it is an outer vertex in $G-e$.
\end{enumerate}
\end{lemma}

By Lemma~\ref{lem:induction}\ref{item:induction-ends}, for every edge $f\in F$, an end of $f$ is cubic in $G-e$ if and only if it is an outer vertex of $G-e$. Since every end of $f$ has the same degree in $G$ as in $G-e$, an end of $f$ is cubic in $G$ if and only if it is an outer vertex of $G-e$. We shall use this result without further mention whenever it is cited later.

\begin{lemma}\label{lem:four-matching}
Let $H[A,B]$ be a simple balanced bipartite graph with $|A|=|B|=4$.
Suppose that the degrees of vertices in $A$ are $3,2,2,2$,  every vertex in $B$ has degree at most $3$, and $H$ has no isolated vertices. Then $H$ has a perfect matching.
\end{lemma}

\begin{proof}
Let $A=\{a_1,a_2,a_3,a_4\}$ such that $d(a_1)=3$. Then $d(a_2)=d(a_3)=d(a_4)=2$.
By Lemma~\ref{lem:hall}, it suffices to prove $|N(X)|\ge|X|$
for every $X\subseteq A$. This is immediate for $X=\emptyset$.
Assume that $X\ne\emptyset$. Since every vertex in $A$ has at least
two neighbours, $|N(X)|\ge2$.
So the result holds when $|X|\le2$.
Since $H$ has no isolated vertices, we have $N(A)=B$ and so $|N(A)|=|B|=|A|$.
This implies that $|N(X)|=|X|$ when $X=A$.

Now assume that $|X|=3$.
If $a_1\in X$, then $|N(X)|\ge |X|$ since $d(a_1)=3$.
If $a_1\notin X$, then $X=\{a_2,a_3,a_4\}$.
Since  $|N(X)|\ge 2$, if $|N(X)|<|X|$, then $|N(X)|=2$.
It follows that $N(a_2)=N(a_3)=N(a_4)$ and $|N(a_2)|=2$.
Since every vertex in $B$ has degree at most $3$,
each vertex in $N(a_2)$ is nonadjacent to $a_1$.
This implies that $a_1$ has at most two neighbours in $B$, contradicting the fact that $d(a_1)=3$.
Thus, we have $|N(X)|\ge |X|$.
\end{proof}

Let $H[A,B]$ be a simple brace with $|A|=|B|=q\ge3$, and let
$uv\in E(H)$, where $u\in A$ and $v\in B$.
A subset $X\subseteq A$ is called \emph{critical} for $uv$ in $H$ if it
satisfies the conditions in~\eqref{eq:critical-set} below.
Similarly, a subset $Y\subseteq B$ is called \emph{critical}  for $uv$ in $H$ if it
satisfies the conditions $1\le |Y|\le q-2$, $v\in Y$, $N_H(u)\cap Y=\{v\}$,  and $|N_H(Y)|=|Y|+2$.

\begin{lemma}\label{lem:critical-set}
Let $H[A,B]$ be a simple brace with $|A|=|B|=q\ge3$, and let $uv$
be an edge of $H$ with $u\in A,v\in B$. Then $H-uv$ is not a brace
if and only if there is an $X\subseteq A$ satisfying
\begin{equation}\label{eq:critical-set}
 1\le|X|\le q-2,\ u\in X,\
 N_H(v)\cap X=\{u\},\ |N_H(X)|=|X|+2.
\end{equation}
For such a set $X$, the set
$Y=B\setminus(N_H(X)\setminus\{v\})$ is a critical set for $uv$
of size $q-|X|-1$ in $H$ (called the \emph{dual critical set} of $X$ in $H$).
\end{lemma}

\begin{proof}
Since $H$ is a brace,  Lemma \ref{lem:brace-3-connected} implies that $H$ is $3$-connected. Then $H-uv$ is connected.
Suppose that there exists a set $X$ satisfying \eqref{eq:critical-set}.
Then $|N_{H-uv}(X)|=|X|+1$. Thus, Lemma~\ref{lem:brace}\ref{item:brace-neighbourhood} implies that $H-uv$ is not a brace.
Conversely, suppose that $H-uv$ is not a brace.
Then, by Lemma \ref{lem:brace}\ref{item:brace-neighbourhood}, there is a subset $X\subseteq A$ with $1\le |X|\le q-2$ such that $|N_{H-uv}(X)|\le|X|+1$.
Since $H$ is a brace, again by Lemma \ref{lem:brace}\ref{item:brace-neighbourhood}, we have $|N_H(X)|\geq|X|+2$.
It follows that $u\in X$, $N_H(v)\cap X=\{u\}$ and $|N_H(X)|=|X|+2$.

For the set $Y$,  we have $v\in Y$ and  $|Y|=q-(|N_H(X)|-1)=q-|X|-1$, and so $1\le|Y|\le q-2$.
Since $u\in X$, all neighbours of $u$ other than $v$ belong to $N_H(X)\setminus\{v\}$, which
implies that $N_H(u)\cap Y=\{v\}$. Note that no vertex of $Y\setminus\{v\}$
has a neighbour in $X$, and $v$ has only the neighbour $u$ in $X$.
Then $N_H(Y)\subseteq(A\setminus X)\cup\{u\}$. Note that $|(A\setminus X)\cup\{u\}|=q-|X|+1=|Y|+2$.
Since $H$ is a brace, Lemma~\ref{lem:brace}\ref{item:brace-neighbourhood} implies that $|N_H(Y)|\ge |Y|+2$.
Thus $N_H(Y)=(A\setminus X)\cup\{u\}$ and $|N_H(Y)|=|Y|+2$.
Hence, $Y$ is critical for $uv$ in $H$. Moreover, we have
$A\setminus(N_H(Y)\setminus\{u\})=X$, so the dual critical set of $Y$ is $X$.
\end{proof}

\begin{lemma}\label{lem:fourteen-critical-set}
Let $G[A,B]$ be a simple brace in which each colour class consists
of four cubic vertices and three vertices of degree $4$.
Then $|E(G[V_3])|=|E(G[V_4])|$.
Furthermore, for an edge $uv$ of $G$, $u\in A,v\in B$, whose ends both have degree $4$, the following statements hold.
\begin{enumerate}[label=(\roman*),nosep]
\item\label{item:fourteen-critical-set-one}
Every critical set for $uv$ has size $2$, $3$ or $4$;

\item\label{item:fourteen-critical-set-two}
A critical set of size $2$ contained in $A$ gives a cubic vertex $w\in A$ with
$N(w)=N(u)\setminus\{v\}$;

\item\label{item:fourteen-critical-set-ends}
A critical set of size $4$ contained in $A$ gives a cubic vertex $z\in B$ with $N(z)=N(v)\setminus\{u\}$.
\end{enumerate}
\end{lemma}

\begin{proof}
By the assumption, we have $|V_3|=8$, $|V_4|=6$, and $2|E(G[V_3])|+|E(V_3,V_4)|=3|V_3|=24$, and $2|E(G[V_4])|+|E(V_3,V_4)|=4|V_4|=24$. It follows that $|E(G[V_3])|=|E(G[V_4])|$.

(i) By definition, a critical set $X\subseteq A$ for $uv$ satisfies $N(v)\cap X=\{u\}$ and
$|N(X)|=|X|+2$.
This implies that $u\in X\subseteq A\setminus\big(N(v)\setminus\{u\}\big)$.
Since $d(u)=4$, we have $X\neq \{u\}$ and so $|X|\ge 2$.
Since $d(v)=4$ and $|A|=7$, we have $|A\setminus\big(N(v)\setminus\{u\}\big)|=4$ and so $|X|\le 4$.
Thus, we have $2\le |X|\le 4$.
The same argument, with $A$ and $B$ interchanged, applies to critical
sets for $uv$ contained in $B$. (i) holds.

(ii) Suppose that $X=\{u,a\}$ is a critical set in $A$. Then $|N(X)|=|X|+2=4$.
Since $d(u)=4$, we have $N(X)=N(u)$. Recall that $v\notin N(a)$.
Then $N(a)\subseteq N(u)\setminus\{v\}$.
It follows that $N(a)=N(u)\setminus\{v\}$ since $d(a)\ge 3$ and $|N(u)\setminus\{v\}|=3$.
Hence, we have $d(a)=3$ and $a$ is the required vertex. (ii) holds.

(iii) Suppose that $X\subseteq A$ is a critical set of size $4$.
By Lemma~\ref{lem:critical-set}, a dual critical set $Y\subseteq B$ of $X$ has size $7-|X|-1=2$.
Applying part~\ref{item:fourteen-critical-set-two} with $A$ and $B$
interchanged gives a cubic vertex $z\in B$ with
$N(z)=N(v)\setminus\{u\}$.  (iii) holds.
\end{proof}

\section{Proof of Theorem~\ref{thm:transversal}}\label{sec:transversal}

Let $D$ be a digraph and $v\in V(D)$. We denote by $d_D^-(v)$ and $d_D^+(v)$ the indegree and outdegree of $v$.
A digraph is \emph{strongly connected}, or \emph{strong}, if each ordered pair of vertices is joined
by a directed path.
A digraph is \emph{$2$-strongly connected}, or \emph{$2$-strong}, if it has at least three
vertices and remains strongly connected after deleting any vertex.
A $2$-strong digraph is \emph{minimal} if the digraph obtained by deleting any arc is not $2$-strong.
For a simple bipartite graph $H[A,B]$ with a perfect matching
$M=\{a_ib_i:1\le i\le h\}$, where $a_i\in A$ and $b_i\in B$, the \emph{associated digraph}
$D(H,M)$ has vertices $1,\ldots,h$ and contains an arc $ij$ for each
edge $a_ib_j$ not in $M$.

\begin{lemma}\label{lem:correspondence}
Let $H$ be a simple bipartite graph of order at least $6$ with a perfect matching $M$.
\begin{enumerate}[label=(\roman*),nosep]
\item\label{item:correspondence-equivalence}(\cite{RST1999})
The graph $H$ is $2$-extendable if and only if $D(H,M)$ is $2$-strong;
\item\label{item:correspondence-minimal}(\cite{ZL})
If $H$ is minimal $2$-extendable, then $D(H,M)$ is
minimal $2$-strong.
\end{enumerate}
\end{lemma}

\begin{lemma}[\cite{Mader}]\label{lem:mader}
Every minimal $2$-strong digraph $D$ has a vertex $v$ with $d_D^+(v)=d_D^-(v)=2$.
\end{lemma}

\begin{proof}[Proof of Theorem~\ref{thm:transversal}]
Suppose that $G[A,B]$ is a minimal brace of order $n\ge6$.
Recall that every minimal brace is simple. Then $G$ is simple.
Let $M=\{a_i b_i:1\le i\le n/2\}$ be any perfect matching
of $G$, where $a_i\in A$ and $b_i\in B$.
By Lemma~\ref{lem:brace-3-connected}, $G$ is $3$-connected,
and so $\delta(G)\ge3$.
Since $G$ is a minimal brace,
Lemma~\ref{lem:brace}\ref{item:brace-extendable} implies
that $G$ is minimal $2$-extendable.
Then, by
Lemma~\ref{lem:correspondence}\ref{item:correspondence-minimal},
$D(G,M)$ is minimal $2$-strong.
Lemma~\ref{lem:mader} implies that $D(G,M)$ has a vertex $i$
whose indegree and outdegree are both $2$.
Writing $D=D(G,M)$, by the definition of $D$, we have
$d_G(a_i)=d_D^+(i)+1$ and $d_G(b_i)=d_D^-(i)+1$.
Hence, $a_i b_i\in M$ is a $3$-line of $G$.
This holds for every perfect matching of $G$.

Since $\delta(G)\ge3$, choose
$u\in N_G(a_i)\setminus\{b_i\}$.
Then $u\in B$.
Since $G$ is matching covered, it has a perfect matching
$M_1$ containing $a_i u$.
Clearly, $a_i b_i\notin M_1$.
By the argument above, $M_1$ contains a $3$-line
$xy\ne a_i b_i$, where $x\in A$ and $y\in B$.
We show that there exists a perfect matching of $G$
which avoids both $a_i b_i$ and $xy$, which contains
a third $3$-line of $G$. Thus, $G$ has at least three $3$-lines.

If $xy=a_i u$, choose
$w\in N_G(a_i)\setminus\{b_i,u\}$ using $\delta(G)\ge3$.
Since $G$ is matching covered, it has a perfect matching of $G$
containing $a_i w$.
This perfect matching avoids both $a_i b_i$ and $xy$.

Now suppose that $xy\ne a_i u$.
Since $xy,a_i u\in M_1$, these two edges are independent in $G$.
In particular, $x\ne a_i$ and $y\ne u$.
If $xy$ and $a_i b_i$ are adjacent, then $y=b_i$.
Since $\delta(G)\ge3$, choose
$w'\in N_G(y)\setminus\{x,a_i\}$.
Since $G$ is matching covered, it has a perfect matching of $G$
containing $yw'$, which avoids both $a_i b_i$ and $xy$.
If $xy$ and $a_i b_i$ are independent, choose
$z\in N_G(y)\setminus\{a_i,x\}$ as $\delta(G)\ge3$.
Since $y\ne u$ and $z\ne a_i$, the edges $a_i u$ and $yz$
are independent in $G$.
Since $G$ is $2$-extendable,
they extend to a perfect matching of $G$.
This matching avoids $a_i b_i$ because it contains $a_i u$,
and avoids $xy$ because it contains $yz$. The proof is complete.
\end{proof}

\section{Proof of Theorem~\ref{thm:main}}\label{sec:main-proof}
Throughout this section,
suppose $G[A,B]$ is a minimal brace of order $n\ge6$ and size $m$.
Then, by Lemma \ref{lem:brace-3-connected}, $G$ is 3-connected and so $\delta(G)\ge 3$.
Recall that $G$ is simple.
Let $t=|E(G[V_3])|$, the number of $3$-lines of $G$, and let $c$
be the number of components of $G[V_{\ge4}]$.
Note that $V_{\geq4}(G)$ is the set of noncubic vertices in $G$.
When there are no noncubic vertices in $G$, set $c=0$.
The total degree excess above four at the noncubic vertices is
$\rho=\sum_{v\in V_{\ge4}}\bigl(d_G(v)-4\bigr)$.

\subsection{Bounds on the number of cubic vertices}\label{sec:count}

\begin{lemma}\label{lem:counts}
Let $G$, $t,c,$ and $\rho$ be as defined above. Then
\begin{align}
 5n_3(G)&=2n+2t+2c+\rho,\label{eq:identity-order}\\
 2n_3(G)&=m-n+t+c.\label{eq:identity-size}
\end{align}
Consequently,
\begin{equation}\label{eq:linear}
 n_3(G)\ge\left\lceil\frac{2n+8}{5}\right\rceil,
 \
 n_3(G)\ge\left\lceil\frac{m-n+4}{2}\right\rceil.
\end{equation}
\end{lemma}

\begin{proof}
Write $r=|V_{\ge4}|=n-n_3(G)$ and $x=|E(V_3,V_{\ge4})|$.
By Lemma~\ref{lem:forest}, $G[V_{\ge4}]$ is a forest, which implies that it has $r-c$ edges.
Note that
\[
 3n_3(G)=2t+x,\ 4r+\rho=2(r-c)+x,\ m=t+x+r-c.
\]
The first equality gives $x=3n_3(G)-2t$. Substituting this into
the last two  equalities and  using $r=n-n_3(G)$, we obtain $2(n-n_3(G))+\rho=3n_3(G)-2t-2c$,
which is~\eqref{eq:identity-order}, and
$m=t+3n_3(G)-2t+n-n_3(G)-c=n+2n_3(G)-t-c$,
which is~\eqref{eq:identity-size}.

If $V_{\ge4}=\emptyset$, then $n_3(G)=n$ and $m=3n/2$.
Since $n\ge6$, we have $5n\ge2n+8$ and $2n\ge n/2+4$,
which give both inequalities in~\eqref{eq:linear}.
Now assume that $V_{\ge4}\ne\emptyset$. Then $c\ge1$ and $\rho\ge0$.
By Theorem~\ref{thm:transversal}, we have $t\ge3$.
Combining these with~\eqref{eq:identity-order} and~\eqref{eq:identity-size}
gives~\eqref{eq:linear}.
\end{proof}

\begin{proof}[Proof of Theorem~\ref{thm:main}\ref{item:count}]
By Lemma~\ref{lem:counts}, we have $n_3(G)\ge\left\lceil\frac{2n+8}{5}\right\rceil$.
Suppose that $n_3(G)\le7$. Then $2n+8\le5n_3(G)\le35$.
Since $n$ is even, we have $n\le12$.
By Lemma~\ref{lem:small}\ref{item:small-six-eight}--\ref{item:small-ten}, $G$ is $K_{3,3}$ when $n=6$, and $n_3(G)\ge 8$ when $n=8$ or $10$.
If $n=12$, then $2m\ge3n_3(G)+4(12-n_3(G))=48-n_3(G)\ge41$, and so $m\ge21$.
This contradicts Lemma~\ref{lem:small}\ref{item:small-twelve}.
Thus $G$ is $K_{3,3}$, which has exactly six cubic vertices. This completes the proof of Theorem~\ref{thm:main}\ref{item:count}.
\end{proof}

\subsection{Reduction to two orders}
We now proceed with the proof of Theorem \ref{thm:main}\ref{item:eight}.
Assume further that $G$ has exactly eight cubic vertices, i.e., $n_3(G)=8$.
The following parameters follow directly from the constructions.

\begin{proposition}\label{prop:parameters}
The four graphs $B_8, B_{10}, Q_{10}^{+}$ and $Q_{12}$ in Figure~\ref{fig:eight-cubic} have the following parameters:
\begin{center}
\normalfont
\setlength{\tabcolsep}{12pt}
\renewcommand{\arraystretch}{1.1}
\begin{tabular}{@{}ccccc@{}}
\toprule
Graph & $n$ & $m$ & $n_3(G)$ & $t$ \\
\midrule
$B_8$       & 8  & 12 & 8 & 12 \\
$B_{10}$    & 10 & 16 & 8 & 8  \\
$Q_{10}^{+}$& 10 & 16 & 8 & 9  \\
$Q_{12}$    & 12 & 20 & 8 & 4  \\
\bottomrule
\end{tabular}
\end{center}
\end{proposition}

\begin{lemma}\label{lem:eight-reduction}
Either $G$ is isomorphic to one of $B_8, B_{10}, Q_{10}^{+}$ and $Q_{12}$, or one of the
following holds.
\begin{enumerate}[label=(\roman*),nosep]
\item\label{item:reduction-fourteen}
$n=14$, each colour class of $G$ consists of four cubic vertices and  three vertices of degree $4$, $m=24$, and $t+c=6$;
\item\label{item:reduction-sixteen}
$n=16$, each colour class of $G$ consists of four cubic vertices and  four vertices of degree $4$, $t=3$, and $G[V_4]$ is a tree.
\end{enumerate}
\end{lemma}

\begin{proof}
Since $n_3(G)=8$, we have $n\ge8$.
By Lemma~\ref{lem:counts}, we have $n_3(G)\ge\left\lceil\frac{2n+8}{5}\right\rceil$
and $n_3(G)\ge\left\lceil\frac{m-n+4}{2}\right\rceil$.
The first inequality, together with $n_3(G)=8$, gives $n\le16$.
Recall that $G$ is a minimal brace.
If $n=8$, Lemma \ref{lem:small}\ref{item:small-six-eight} implies that $G$ is isomorphic to $B_8$.
If $n=10$,  Lemma~\ref{lem:small}\ref{item:small-ten} implies that $G$ is isomorphic to one of $ B_{10}$ and $Q_{10}^{+}$.
If $n=12$, then $m\ge \frac{8\times3+4\times4}{2}=20$.
By Lemma~\ref{lem:small}\ref{item:small-twelve}, $G$ is isomorphic to $Q_{12}$.

Suppose first that $n=14$. Then $m=\frac{8\times3+6\times4+\rho}{2}=24+\rho/2$.
By Lemma~\ref{lem:size}, we have $m\le25$, with equality if and only if $G\cong Q_{14}$.
However, $Q_{14}$ has exactly ten cubic vertices.
It follows that $m=24$ and so $\rho=0$.
Thus, every noncubic vertex of $G$ has degree exactly $4$.
Since $m=24$, by a simple calculation, we can derive that each of $A$ and $B$
consists of four cubic vertices and  three vertices of degree $4$.
Using equation~\eqref{eq:identity-size} of Lemma~\ref{lem:counts}, we have $2n_3(G)=m-n+t+c$, which implies that $t+c=6$.

Suppose now that $n=16$. Then using the equation~\eqref{eq:identity-order} of
Lemma~\ref{lem:counts}, we have $5n_3(G)=2n+2t+2c+\rho$. So $2t+2c+\rho=8$.
By Theorem~\ref{thm:transversal}, $G$ has at least three $3$-lines, i.e., $t\ge3$.
Since $n=16$ and $n_3(G)=8$, we have $V_{\ge4}\ne\emptyset$, which implies that
$c\ge1$ and $\rho\ge0$.
Thus, we have $t=3$, $c=1$, and $\rho=0$.
The equality $\rho=0$ implies that every noncubic vertex of $G$ has degree exactly $4$.
Thus $m=\frac{8\times3+8\times4}{2}=28$.
By a simple calculation, we can derive that each of $A$ and $B$
consists of four cubic vertices and four vertices of degree $4$.
By Lemma~\ref{lem:forest}, $G[V_4]$ is a forest.
Combining this with $c=1$, we conclude that $G[V_4]$ is a tree.
\end{proof}

\begin{lemma}\label{lem:three-shapes}
If a minimal brace $H$ of order at least $6$ has exactly three $3$-lines, then the graph obtained by deleting the
isolated vertices of $H[V_3(H)]$ is either $K_{1,3}$ or $3K_2$.
\end{lemma}

\begin{proof}
If the three $3$-lines are pairwise disjoint, then deleting the
isolated vertices of $H[V_3(H)]$ yields $3K_2$, as required.
So we may assume that $va,vb$ are two $3$-lines of $H$.
Let $xy$ be the third $3$-line of $H$ such that $x$ and $v$ lie in the same colour class of $H$.
Then $\{v,a,b,x,y\}\subseteq V_3(H)$.

Suppose that $x\ne v$.
Since $d_H(v)=3$, let $N_H(v)=\{a,b,w\}$.
As $d_H(x)=3$, choose $z\in N_H(x)\setminus\{y,w\}$.
Note that the edges $vw$ and $xz$ are independent.
By Lemma~\ref{lem:brace}\ref{item:brace-extendable}, $H$ is $2$-extendable, which implies that
there is a perfect matching $M$ of $H$ containing both $vw$ and $xz$.
Note that $va,vb,xy\notin M$.
Then $M$ contains no $3$-line, contradicting Theorem~\ref{thm:transversal} which states that every perfect matching of $H$ contains a $3$-line. Therefore $x=v$, and deleting the isolated vertices of $H[V_3(H)]$ yields $K_{1,3}$. The result holds.
\end{proof}

\subsection{Excluding order sixteen}

\begin{lemma}\label{lem:sixteen}
$n\neq 16$.
\end{lemma}

\begin{proof}
Suppose, to the contrary, that $n=16$.
By Lemma~\ref{lem:eight-reduction}\ref{item:reduction-sixteen},
each colour class of $G$ consists of four cubic vertices and  four vertices of degree $4$, $t=3$, and $G[V_4]$ is a tree.

If the three edges of $G[V_3]$ form a star, then the other three
cubic vertices in the center's colour class have at most four
neighbours, all in $V_4$. For this three-element set,
Lemma~\ref{lem:brace}\ref{item:brace-neighbourhood}
requires at least five neighbours, a contradiction.
By Lemma~\ref{lem:three-shapes}, deleting the isolated vertices
of $G[V_3]$  yields $3K_2$. We may label the cubic vertices such that
$V_3\cap A=\{a_0,a_1,a_2,a_3\}$, $V_3\cap B=\{b_0,b_1,b_2,b_3\}$, and $E(G[V_3])=\{a_ib_i:1\le i\le3\}$.

Consider the two balanced bipartite subgraphs $H_1=G[(V_3\cap A)\cup(V_4\cap B)]$ and $H_2=G[(V_3\cap B)\cup(V_4\cap A)]$.
Then $d_{H_1}(a_0)=d_{H_2}(b_0)=3$ and $d_{H_1}(a_i)=d_{H_2}(b_i)=2$, $1\le i\le3$.
Since $G[V_4]$ is a tree, for each $w\in V_4\cap B$, $d_{H_1}(w)=4-d_{G[V_4]}(w)\le3$;
likewise, for each $w'\in V_4\cap A$,
$d_{H_2}(w')=4-d_{G[V_4]}(w')\le3$.

If neither $H_1$ nor $H_2$ had an isolated vertex,
Lemma~\ref{lem:four-matching} would give perfect matchings $M_1$
and $M_2$ in $H_1$ and $H_2$, respectively.
Since $V(H_1)$ and $V(H_2)$ partition $V(G)$,
$M_1\cup M_2$ is a perfect matching of $G$.
Every edge of this perfect matching joins a cubic vertex to a noncubic
vertex in $G$, so it contains no $3$-line of $G$, contradicting Theorem~\ref{thm:transversal}.

Without loss of generality, assume that $H_2$ has an isolated vertex $u$.
Then  $u\in V_4\cap A$, $d_{G[V_4]}(u)=d_G(u)=4$ and $N_G(u)=V_4\cap B$.
Since $G[V_4]$ is a tree on eight vertices, it has seven edges.
Consequently, $\sum_{w'\in(V_4\cap A)\setminus\{u\}}d_{G[V_4]}(w')=7-4=3$.
It follows that each vertex of $(V_4\cap A)\setminus\{u\}$ is a leaf in $G[V_4]$,
and so has exactly three neighbours in $V_3\cap B$.
Since $N_G(b_0)\subseteq V_4\cap A$ and $ub_0\notin E(G)$, we have $N_G(b_0)=(V_4\cap A)\setminus\{u\}$.
For $1\le i\le3$, we have $N_G(b_i)\cap V_3=\{a_i\}$, $N_G(b_i)\setminus\{a_i\}\subseteq(V_4\cap A)\setminus\{u\}$.
Thus, each vertex of $(V_4\cap A)\setminus\{u\}$ is adjacent to $b_0$ and exactly two of
$b_1,b_2,b_3$, while each $b_i$, $1\le i\le3$, is adjacent
to exactly two vertices of $(V_4\cap A)\setminus\{u\}$ in $G$.
Consequently, $\bigl(N_G(w)\cup N_G(w')\bigr)\cap V_3=V_3\cap B$
for any distinct $w,w'\in(V_4\cap A)\setminus\{u\}$ as $G$ is simple.

Choose $v\in N_G(a_0)\subseteq V_4\cap B$.  Then $uv\in E(G)$.
We will show that there is no critical set contained in $A$ for $uv$ in $G$. Then, 
by Lemma~\ref{lem:critical-set}, $G-uv$ is a brace,
a contradiction  to the minimality of $G$. Therefore, we have $n\ne 16$.

Suppose, to the contrary, that $X\subseteq A$ is a critical set for $uv$ in $G$.
By  definition, we have
$$
u\in X,\ N_G(v)\cap X=\{u\},\
|N_G(X)|=|X|+2.
$$
Since $d_G(v)=4$ and $|A|=8$, we have $X\subseteq A\setminus\bigl(N_G(v)\setminus\{u\}\bigr)$
and $|X|\le5$.
Moreover, $a_0\notin X$ since $a_0\in N_G(v)\setminus\{u\}$.
We distinguish the following three cases according to $|X\cap V_4|$.

If $|X\cap V_4|\ge3$, then $X$ contains $u$ and at least
two vertices of $(V_4\cap A)\setminus\{u\}$. Since $N_G(u)=V_4\cap B$ and any
two vertices of $(V_4\cap A)\setminus\{u\}$ together have all of $V_3\cap B$ as neighbours,
we have $N_G(X)=B$ and $|X|=|N_G(X)|-2=6$, a contradiction  to $|X|\le5$.

If $|X\cap V_4|=1$, then $X\cap V_4=\{u\}$.
As $a_0\notin X$, every vertex of $X\setminus\{u\}$ is
some $a_i$ with $1\le i\le3$, and contributes exactly the
additional neighbour $b_i$ outside $N_G(u)$.
Therefore $N_G(X)=(V_4\cap B)\cup\{b_i:a_i\in X,\ 1\le i\le3\}$,
and hence $|N_G(X)|=4+(|X|-1)=|X|+3$,
a contradiction  to $|N_G(X)|=|X|+2$.

Finally, suppose that $|X\cap V_4|=2$.
Besides $u$, the set $X$ contains exactly one vertex of $(V_4\cap A)\setminus\{u\}$, which
misses $b_j$ for some $1\le j\le3$. Thus $(V_4\cap B)\cup\bigl((V_3\cap B)\setminus\{b_j\}\bigr)\subseteq N_G(X)$.
It follows that $7\le |N_G(X)|=|X|+2\le7$.
Hence $|X|=5$, $|N_G(X)|=7$, and $b_j\notin N_G(X)$.
Since $a_jb_j\in E(G)$, we have $a_j\notin X$.
Together with $a_0\notin X$, we have $X\cap V_3\subseteq(V_3\cap A)\setminus\{a_0,a_j\}$ and so
$|X\cap V_3|\le2$.
But $|X\cap V_3|=|X|-|X\cap V_4|=5-2=3$, a contradiction.
The proof is complete.
\end{proof}

\subsection{Excluding order fourteen}

We now prove that $n\neq 14$. To establish this conclusion, we first show that the graph $G$ cannot have the following two structures when $n=14$.
Recall that $t=|E(G[V_3])|$ is the number of $3$-lines of $G$, and  $c$ is the number of components of $G[V_{\ge4}]$.

\begin{lemma}\label{lem:fourteen-not}
If $n=14$, then neither of the following two statements can hold.
\begin{enumerate}[label=(\roman*),nosep]
\item\label{item:fourteen-not-three}
$|E(G[V_3])|=|E(G[V_4])|=3$, and
\item\label{item:fourteen-not-double}
$G[V_4]$ is a double star.
\end{enumerate}
\end{lemma}

\begin{proof}
By Lemma~\ref{lem:eight-reduction}\ref{item:reduction-fourteen},
each colour class of $G$ consists of four cubic vertices and  three vertices of degree $4$, $m=24$, and $t+c=6$. Then $V(G)=V_3\cup V_4$. By Lemma \ref{lem:forest}, $G[V_4]$ is a forest.
So $|E(G[V_4])|=6-c$.

(i) Suppose that $|E(G[V_3])|=|E(G[V_4])|=3$. Then $t=c=3$.
If distinct vertices $t_1,t_2\in V_3\cap A$ are isolated in $G[V_3]$,
then $N_G(\{t_1,t_2\})\subseteq V_4\cap B$.
Since $d_G(t_1)=d_G(t_2)=3$ and $|V_4\cap B|=3$, we have $N_G(t_1)=N_G(t_2)=V_4\cap B$.
Thus $|N_G(\{t_1,t_2\})|=3$, whereas
Lemma~\ref{lem:brace}\ref{item:brace-neighbourhood}
requires $|N_G(\{t_1,t_2\})|\ge4$ since $G$ is a brace, a contradiction.
The same argument holds for $V_3\cap B$.
Thus, the three edges of $G[V_3]$ have three distinct ends
in each colour class and hence form a matching in $G$.
Let $q\in V_3\cap A$ and $p\in V_3\cap B$ be isolated
vertices of $G[V_3]$. Then $N_G(p)=V_4\cap A$ and $N_G(q)=V_4\cap B$.
Since $|E(G[V_4])|=3$, we may assume that $uv\in E(G[V_4])$ with $u\in A$ and $v\in B$. Then $qv,pu\in E(G)$.

We will show that  there is   no critical set contained in $A$ for $uv$ in $G$. Then,
by Lemma~\ref{lem:critical-set}, $G-uv$ is a brace, contradicting the fact that $G$ is a minimal brace.

Suppose, to the contrary, that $X\subseteq A$ is a critical set for $uv$ in $G$.
By definition, we have $u\in X$, $N_G(v)\cap X=\{u\}$ and $|N_G(X)|=|X|+2$.
By Lemma~\ref{lem:fourteen-critical-set}\ref{item:fourteen-critical-set-one}, we have
$|X|\in\{2,3,4\}$.

If $|X|=2$, Lemma~\ref{lem:fourteen-critical-set}\ref{item:fourteen-critical-set-two} implies that there exists a cubic vertex $w\in A$ with
$N_G(w)=N_G(u)\setminus\{v\}$.
Since $p\in N_G(u)\setminus\{v\}$, we obtain
$wp\in E(G[V_3])$, contradicting the fact that $p$ is an isolated vertex of $G[V_3]$.
If $|X|=4$, Lemma~\ref{lem:fourteen-critical-set}\ref{item:fourteen-critical-set-ends}  implies that there exists a cubic vertex $z\in B$ with $N_G(z)=N_G(v)\setminus\{u\}$.
Since $q\in N_G(v)\setminus\{u\}$,
we have $zq\in E(G[V_3])$, contradicting the fact that $q$ is an isolated vertex of $G[V_3]$.

If $|X|=3$, then $|N_G(X)|=5$.
Let $X=\{u,x,y\}$. Then $xv,yv\notin E(G)$ as $N_G(v)\cap X=\{u\}$.
Since $qv\in E(G)$, we have $q\notin X$.
This implies that, in $G$, every cubic vertex in $X$ has exactly one cubic neighbour
and exactly two  neighbours of degree $4$ in $B$.
We next consider three cases depending on whether $x$ and $y$ belong to $V_3$.

If $x,y\in V_3$, since $G[V_3]$ is a matching with $p$ isolated,
their cubic neighbours in $G$ are distinct and different from $p$.
Recall that  $xv,yv\notin E(G)$. Then both $x$ and $y$ have noncubic neighbourhood
$(V_4\cap B)\setminus\{v\}$ in $G$.
Together with $p,v\in N_G(u)$, we have $|N_G(X)|\ge6$, contradicting the fact that $|N_G(X)|=5$.

If exactly one of $x$ and $y$ belongs to $V_4$, say $x$, then  $y\in V_3$.
Since $d_G(x)=4$, $v\notin N_G(x)$, and $|N_G(X)|=5$,
we have $N_G(x)=N_G(X)\setminus\{v\}$.
Since $qv\in E(G)$ and $yv\notin E(G)$, we have $y\ne q$.
Since the three edges of $G[V_3]$ form a matching and $q$ is the
only isolated vertex in $V_3\cap A$, $y$ has exactly one neighbour in $V_3\cap B$ and exactly two neighbours in $(V_4\cap B)\setminus\{v\}$.
The cubic neighbour of $y$ differs from $p$ since $p$ is isolated in
$G[V_3]$. Since $p\in N_G(u)$ and
$N_G(x)=N_G(X)\setminus\{v\}$, $N_G(x)$ contains the two vertices of $(V_4\cap B)\setminus\{v\}$, the vertex $p$, and the cubic neighbour of $y$.
Hence $d_{G[V_4]}(x)=2$ and $|N_G(x)\cap V_3|=2$.
As $|E(G[V_4])|=3$ and $uv\in E(G[V_4])$,
we have $d_{G[V_4]}(u)=1$.
Combining this with $N_G(x)=N_G(X)\setminus\{v\}$, we obtain $N_G(u)\setminus\{v\}\subseteq N_G(x)\cap V_3$. However, the left-hand set has size three and the right-hand
set has size two, a contradiction.

Finally, suppose that $x,y\in V_4$.
Then $V_4\cap A=\{u,x,y\}$ and
$N_G(x)=N_G(y)=N_G(X)\setminus\{v\}$ because $|N_G(X)|=5$.
Consequently, $d_{G[V_4]}(x)=d_{G[V_4]}(y)$ and
$d_{G[V_4]}(u)+2d_{G[V_4]}(x)=3$.
Since $d_{G[V_4]}(u)\ge1$, we have
$d_{G[V_4]}(x)\in\{0,1\}$.
If $d_{G[V_4]}(x)=0$, then
$d_{G[V_4]}(u)=3$ and $N_G(x)=V_3\cap B$.
Thus, $N_G(u)\setminus\{v\}$ contains two vertices in $V_4$. However,  $N_G(u)\setminus\{v\}\subseteq N_G(x)=V_3\cap B$ , a contradiction.
If $d_{G[V_4]}(x)=1$, then $d_{G[V_4]}(u)=1$.
Since $N_G(u)\setminus\{v\}\subseteq N_G(x)=N_G(y)$,
the vertices $u,x,y$ have the same three neighbours in $V_3\cap B$.
Then at least two of these neighbours differ from $p$, and
each of them is adjacent to $u,x,y$ and also has a neighbour
in $G[V_3]$, contradicting its degree three in $G$.

(ii)
Suppose that $G[V_4]$ is a double star. Then $c=1$ and $|E(G[V_4])|=6-c=5$.
Since $|V_4\cap A|=|V_4\cap B|=3$, we may label its
vertices so that $V_4\cap A=\{u,a_1,a_2\}$,
$V_4\cap B=\{v,b_1,b_2\}$, and $E(G[V_4])=\{uv,a_1v,a_2v,ub_1,ub_2\}$.
Let $p\in V_3\cap B$ and $q\in V_3\cap A$ be the unique
cubic neighbours of $u$ and $v$ in $G$, respectively.
Then $N_G(u)=\{v,b_1,b_2,p\}$ and
$N_G(v)=\{u,a_1,a_2,q\}$.
For $i=1,2$, let $r_i$ be the unique vertex of $V_3\cap B$
not adjacent to $a_i$. Thus $a_i$ is adjacent to $v$
and all vertices of $(V_3\cap B)\setminus\{r_i\}$.

The vertex $p$ is adjacent to $u$ and to precisely those
$a_i$ with $r_i\ne p$ in $V_4$. Every other vertex of $V_3\cap B$
has neighbours in $V_4$ only among $a_1,a_2$.
This implies that
$d_{G[V_3]}(p)=|\{i\in\{1,2\}:r_i=p\}|$ and
$d_{G[V_3]}(r)=1+|\{i\in\{1,2\}:r_i=r\}|$
for every $r\in(V_3\cap B)\setminus\{p\}$.

We will show that there is no critical set contained in $A$ for at least one of $a_1v$ and $a_2v$ in $G$.
Then, by Lemma~\ref{lem:critical-set}, at least one of $G-a_1v$ and $G-a_2v$ is a brace, contradicting the fact that $G$ is a minimal brace.

Suppose, to the contrary, that $X\subseteq A$ is a critical set for $a_1v$ in $G$.
By definition, we have $a_1\in X$, $N_G(v)\cap X=\{a_1\}$ and
$|N_G(X)|=|X|+2$. In particular,
$X\setminus\{a_1\}\subseteq(V_3\cap A)\setminus\{q\}$.
By Lemma~\ref{lem:fourteen-critical-set}\ref{item:fourteen-critical-set-one}, we have
$|X|\in\{2,3,4\}$.

If $|X|=3$, then each of $b_1$ and $b_2$ is adjacent to at least one
of the two vertices in $X\setminus\{a_1\}$ since each
has three neighbours in  $V_3\cap A$.
Thus $N_G(X)$ contains $v,b_1,b_2$ and the three neighbours
of $a_1$ in $V_3\cap B$.
Hence $|N_G(X)|\ge6>|X|+2$, a contradiction.

If $|X|=2$, then $|N_G(X)|=4$.
Write $X=\{a_1,w\}$. Then $N_G(X)=N_G(a_1)$,  $w\in (V_3\cap A)\setminus\{q\}$ and $wv\notin E(G)$.
So $N_G(w)\subseteq N_G(a_1)\setminus\{v\}$.
Combining this with $d_G(w)=3$, we have
$N_G(w)=N_G(a_1)\setminus\{v\}=(V_3\cap B)\setminus\{r_1\}$.
Thus $w$ is adjacent to neither $b_1$ nor $b_2$,
and each of  $b_1$ and $b_2$ is adjacent to all three vertices
of $(V_3\cap A)\setminus\{w\}$ in $G$.
Since $q\ne w$, we have $N_G(q)=\{v,b_1,b_2\}$.
Write $(V_3\cap A)\setminus\{q,w\}=\{z_1,z_2\}$.
Each $z_j$ is adjacent to $b_1,b_2$ but not to $v$,
so $d_{G[V_3]}(z_j)=1$.
If $z_jp\in E(G)$, then $N_G(z_j)=\{b_1,b_2,p\}$
and $N_G(\{u,q,z_j\})=\{v,b_1,b_2,p\}$,
contradicting Lemma~\ref{lem:brace}\ref{item:brace-neighbourhood} since $G$ is a brace.
So $z_jp\notin E(G)$.
Since $N_G(q)=\{v,b_1,b_2\}$, we conclude that
$N_G(p)\cap V_3\subseteq\{w\}$ and $d_{G[V_3]}(p)\le1$.
If $d_{G[V_3]}(p)=0$, then $p\notin N_G(w)$.
Since $N_G(w)=(V_3\cap B)\setminus\{r_1\}$, we have $r_1=p$.
The formula for $d_{G[V_3]}(p)$ then gives
$d_{G[V_3]}(p)\ge1$, a contradiction.
If $d_{G[V_3]}(p)=1$, then $pw\in E(G)$ and $r_1\ne p$.
The formula for $d_{G[V_3]}(r_1)$ gives
$d_{G[V_3]}(r_1)\ge2$.
Since $N_G(w)=(V_3\cap B)\setminus\{r_1\}$,
we have $r_1w\notin E(G)$.
Combining this with $r_1q\notin E(G)$, we derive that both
$z_1$ and $z_2$ must be adjacent to $r_1$ in $G$.
Hence $N_G(z_1)=N_G(z_2)=\{b_1,b_2,r_1\}$,
contradicting Lemma~\ref{lem:brace}\ref{item:brace-neighbourhood} since $G$ is a brace.

Now we consider $|X|=4$.
By Lemma~\ref{lem:fourteen-critical-set}\ref{item:fourteen-critical-set-ends},  there exists a cubic vertex in $B$ with neighbourhood $N_G(v)\setminus\{a_1\}=\{u,a_2,q\}$.
This vertex must be $p$ which is the unique neighbour of $u$ in $V_3$.
So $N_G(p)=\{u,a_2,q\}$.
Interchanging $a_1$ and $a_2$, we can conclude that  the existence
of a critical set for $a_2v$ in $G$ forces $N_G(p)=\{u,a_1,q\}$.
Since $a_1\ne a_2$, these equalities cannot both hold. The proof is complete.
\end{proof}

We now complete the exclusion of order fourteen using Lemma~\ref{lem:induction}.
Recall that $\mathcal G$ is the set that consists of $K_2,C_4$, all the biwheels, all the
prisms and  M\"obius ladders.
The graphs in the latter two families are all cubic.

\begin{lemma}\label{lem:fourteen}
$n\ne 14$.
\end{lemma}

\begin{proof}
Suppose, to the contrary, that $n=14$.
By Lemma~\ref{lem:eight-reduction}\ref{item:reduction-fourteen},
each colour class of $G$ consists of four cubic vertices and  three vertices of degree $4$, $m=24$, and $t+c=6$. Then $c\ge1$.

Since $\Delta(G)=4$, $G$ is neither a prism nor a M\"obius
ladder, both of which are cubic, and is not $B_{14}$,
which has vertices of degree $6$.
Thus $G\notin \mathcal G$.
By Lemma~\ref{lem:induction}, there is a minimal brace $J=\widehat{G-e}-F$
with index $k\in\{1,2\}$ satisfying $|V(J)|=14-2k$, $|E(J)|=23-2k-|F|$, and $|F|\le k+1$,
where $e$ is a strictly thin edge of $G$.
It follows that $\widehat{G-e}$ is a simple brace.
Since $J$ is a brace,
Lemma~\ref{lem:brace-3-connected} implies that $\delta(J)\ge3$.

If $k=1$, then $|V(J)|=12$, $|E(J)|=21-|F|$, and $|F|\le2$.
By Lemma~\ref{lem:small}\ref{item:small-twelve}, we have $|E(J)|\le20$ and so $|F|\ge 1$.
Thus, we have $1\le |F|\le2$.
If $k=2$, then $|V(J)|=10$, $|E(J)|=19-|F|$, and $|F|\le3$.
By Lemma~\ref{lem:small}\ref{item:small-ten}, we have $|E(J)|\le16$ and so $|F|\ge 3$.
Thus, we have $|F|=3$.
Therefore, we  consider only the following three cases.
In each case, we can obtain a contradiction, which proves that $n\ne14$. 

\medskip
\textbf{Case 1:} $k=1$ and $|F|=1$.
Then $|V(J)|=12$ and $|E(J)|=20$.
By Lemma~\ref{lem:small}\ref{item:small-twelve},
$J$ is  isomorphic to either $B_{12}$ or $Q_{12}$.
\smallskip

Since $k=1$,  $G-e$ has exactly one vertex of degree $2$, which is an end of $e$ since $\delta(G)\ge 3$.
Write $e=uv$, where $d_G(v)=3$ and $d_G(u)=4$, and let
$x_{11},x_{12}$ be the other two neighbours of $v$.
Thus $v$ is the inner vertex and $x_{11},x_{12}$ are the
outer vertices  in $G-e$.

By Lemma~\ref{lem:induction}\ref{item:induction-one},
the unique edge of $F$ is incident  with an outer
vertex in $G-e$, say $x_{11}$. Write $F=\{x_{11}r\}$.
Since $F$ contains only edges surviving bicontraction,  we have $r\ne v$.
Also, $r$ lies in the colour class opposite to $x_{11},x_{12}$,
so $r$ is not an outer vertex in $G-e$.
By Lemma~\ref{lem:induction}\ref{item:induction-ends},
an end of $x_{11}r$ is cubic in $G$ if and only if it is an outer vertex in $G-e$.
So $d_G(x_{11})=3$ and $d_G(r)=4$ since $\Delta(G)=4$.
Let $x'$ be the vertex resulting from bicontraction of $v$ in $G-e$.
Since $\widehat{G-e}$ is simple,
we have $N_G(x_{11})\cap N_G(x_{12})=\{v\}$.
In $\widehat{G-e}$, the edge $x_{11}r$ becomes $x'r$, and they are identified as the same edge.
By the definition of $J$, we have
$d_J(x')=d_G(x_{11})+d_G(x_{12})-3=d_G(x_{12})$ and $d_J(u)=d_J(r)=3$.
All other vertices retain their degrees, so $\Delta(J)\le4$.
Since the hubs of $B_{12}$ have degree $5$, $J$ is isomorphic to $Q_{12}$.

Label $J$ so that
$A_J=\{p_1,p_2,x_1,x_2,x_3,x_4\}$ and
$B_J=\{q_1,q_2,y_1,y_2,y_3,y_4\}$,
with edges $p_\ell y_i$, $x_iq_\ell$, and $x_iy_i$
for $\ell\in\{1,2\}$ and $1\le i\le4$.
Choose the colour classes so that $x'\in A_J$. Then $u\in A_J$ and $r\in B_J$.
Since $\widehat{G-e}$ is simple and $F=\{x_{11}r\}$, we have $x'r\notin E(J)$.
Since every vertex of degree $4$ in $A_J$ is adjacent to every
cubic vertex of $B_J$ in $J$ and $d_J(r)=3$, we have $d_J(x')=3$.
So $d_G(x_{12})=d_J(x')=3$.

Permuting the four pairs $x_i,y_i$, we may assume that
$x'=x_1$ and $r=y_2$, since $r$ is a cubic nonneighbour of $x'$ in $J$.
Since $u\notin\{v,x_{11},x_{12}\}$, it remains distinct
from $x'$ after bicontraction.
Since $u\in A_J$ and $d_J(u)=3$,  we have $u=x_j$ for some $j\in\{2,3,4\}$.
Thus $e=x_jv$ and $F=\{x_{11}y_2\}$.
To reconstruct $G$, add $x_1y_2$ to $J$, split $x_1$
into $x_{11},x_{12}$, and restore $v$ adjacent to
$x_{11},x_{12},x_j$.
Then $x_{11},x_{12}\in A_G$ and $v\in B_G$.
Since $d_G(x_{11})=d_G(x_{12})=3$, each of them has exactly two
neighbours in $\{q_1,q_2,y_1,y_2\}$.
Since $F=\{x_{11}y_2\}$, the vertex $y_2$ is adjacent
to $x_{11}$ in $G$. Up to interchanging $q_1,q_2$, there are
only two cases, as follows.

If both $q_1$ and $q_2$ are adjacent to the same split vertex in $G$, then
$N_G(x_{11})=\{y_1,y_2,v\}$ and
$N_G(x_{12})=\{q_1,q_2,v\}$.
Take $i=1$. Then $q_1x_{1i}\notin E(G)$ and
$N_G(x_{1i})\ne\{q_2,y_j,v\}$.
Figure~\ref{fig:case1} illustrates this case with $j=2$.
If $q_1$ and $q_2$  are adjacent to different split vertices in $G$,
label $q_1$ and $q_2$ such that  $q_1x_{11}\in E(G)$.
Then $N_G(x_{11})=\{q_1,y_2,v\}$ and
$N_G(x_{12})=\{q_2,y_1,v\}$.
Take $i=2$. Again $q_1x_{1i}\notin E(G)$ and
$N_G(x_{1i})\ne\{q_2,y_j,v\}$ since
$j\in\{2,3,4\}$.
Thus, in either case, we can choose $q_1$ and
$i\in\{1,2\}$ so that $q_1x_{1i}\notin E(G)$ and
$N_G(x_{1i})\ne\{q_2,y_j,v\}$.

\begin{figure}[htbp]
\centering
\includegraphics[width=0.6\textwidth]{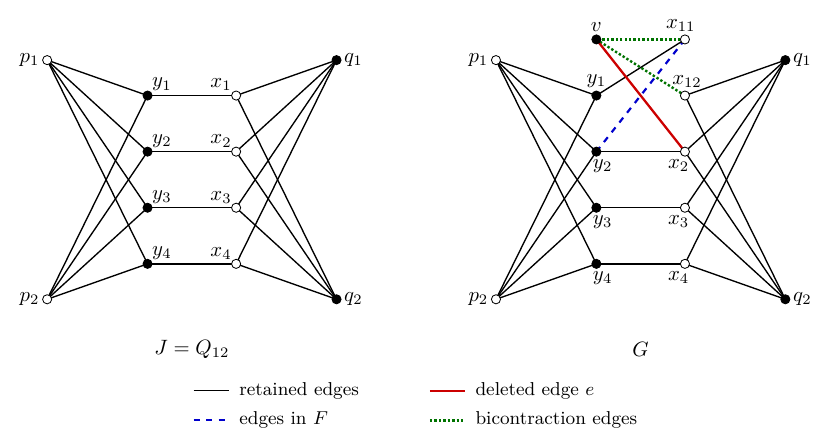}
\caption{Case 1 with $j=2$ of Lemma~\ref{lem:fourteen}:
both $q_1$ and $q_2$ are adjacent to $x_{12}$.}
\label{fig:case1}
\end{figure}

Note that $x_jq_1\in E(G)$ since $j\ne1$.
Suppose that  there exists a critical set $X\subseteq A_G$ for $x_jq_1$ in $G$.
By definition, we have $x_j\in X$, $N_G(q_1)\cap X=\{x_j\}$ and $|N_G(X)|=|X|+2$.
By our choice of $q_1$ and $i$, the vertices in $A_G$
not adjacent to $q_1$ are precisely $p_1,p_2,x_{1i}$.
Thus $x_j\in X\subseteq\{x_j,p_1,p_2,x_{1i}\}$ and so  $|X|\le4$.
Note that $N_G(x_j)=\{q_1,q_2,y_j,v\}$ and $N_G(p_1)=N_G(p_2)=\{y_1,y_2,y_3,y_4\}$.

If $X=\{x_j\}$, then $|N_G(X)|=4>|X|+2$, a contradiction.
If $X$ contains $p_1$ or $p_2$, then $\{q_1,q_2,v,y_1,y_2,y_3,y_4\}\subseteq N_G(X)$.
Thus $|N_G(X)|=7>|X|+2$ since $|X|\le4$, a contradiction.

It remains to consider $X=\{x_j,x_{1i}\}$.  Note that
$N_G(x_j)\setminus\{q_1\}=\{q_2,y_j,v\}$.
By our choice of $q_1$ and $i$,
$q_1\notin N_G(x_{1i})$ and
$N_G(x_{1i})\ne\{q_2,y_j,v\}$.
Since both $N_G(x_{1i})$ and $\{q_2,y_j,v\}$ have size three,
their union has at least four vertices.
Together with $q_1\in N_G(x_j)$, this gives
$|N_G(X)|\ge5>|X|+2$, a contradiction.
Hence, there is no critical set contained in $A_G$ for $x_jq_1$ in $G$.
By Lemma~\ref{lem:critical-set}, $G-x_jq_1$ is a brace,
contradicting the minimality of $G$.

\medskip
\noindent\textbf{Case 2:} $k=2$ and  $|F|=3$.
Then $|V(J)|=10$ and $|E(J)|=16$.
By Lemma~\ref{lem:small}\ref{item:small-ten},
$J$ is isomorphic to $B_{10}$ or $Q_{10}^{+}$.
\smallskip

Since $k=2$, $G-e$ has exactly two vertices of degree $2$, which are ends of $e$ since $\delta(G)\ge 3$.
Write $e=uv$, where $d_G(u)=d_G(v)=3$, and let
$x_{11},x_{12}$ be the other two neighbours of $u$ and
$y_{21},y_{22}$ the other two neighbours of $v$.
Thus, $u,v$ are the inner vertices, and
$x_{11},x_{12},y_{21},y_{22}$ are the outer vertices in $G-e$.
Moreover, since $G$ is bipartite, we have $x_{11}x_{12},y_{21}y_{22}\notin E(G)$, and $N_G(u)\cap N_G(v)=\emptyset$. Therefore, $x_{11}$, $x_{12}$, $y_{21}$ and $y_{22}$ are four distinct vertices.

Since $\widehat{G-e}$ is simple, we have
$N_G(x_{11})\cap N_G(x_{12})=\{u\}$ and
$N_G(y_{21})\cap N_G(y_{22})=\{v\}$.
By Lemma~\ref{lem:induction}\ref{item:induction-two},
there are two adjacent outer vertices, say $x_{11}$ and $y_{21}$,
such that $F\subseteq
\partial_G(x_{11})\cup\partial_G(y_{21})$.
If $F\cap \partial_G(x_{11})\ne \emptyset$, then by Lemma \ref{lem:induction}\ref{item:induction-ends},  we have
$d_G(x_{11})=3$ and $ux_{11}\notin F$ since $d_G(u)=3$ and $u$ is an inner vertex.
Since $|F|=3$, we have $F\cap \partial_G(y_{21})\ne \emptyset$.
By symmetry, if $F\cap \partial_G(y_{21})\ne \emptyset$, then $F\cap \partial_G(x_{11})\ne \emptyset$. It follows that $F\cap \partial_G(x_{11})\ne \emptyset$ and $F\cap \partial_G(y_{21})\ne \emptyset$.
Again by Lemma~\ref{lem:induction}\ref{item:induction-ends},  we have
$d_G(x_{11})=d_G(y_{21})=3$.
Since $|F|=3$, we have $F=\big(\partial_G(x_{11})\cup\partial_G(y_{21})\big)\setminus\{ux_{11},vy_{21}\}$.
Let $N_G(x_{11})=\{u,y_{21},r\}$ and
$N_G(y_{21})=\{v,x_{11},s\}$.
Then $F=\{x_{11}y_{21},x_{11}r,y_{21}s\}$.
Since $\widehat{G-e}$ is simple,  we have
$E_G(\{x_{11},x_{12}\},\{y_{21},y_{22}\})=\{x_{11}y_{21}\}$,
which implies that neither $r$ nor $s$ is an outer vertex.
By Lemma~\ref{lem:induction}\ref{item:induction-ends},
we have $d_G(r)=d_G(s)=4$ since $\Delta(G)=4$.

Let $x'$ and $y'$ be the vertices obtained by contracting
$\{u,x_{11},x_{12}\}$ and $\{v,y_{21},y_{22}\}$,
respectively. The edges $x_{11}y_{21},x_{11}r,y_{21}s$ become
$x'y',x'r,y's$ in $\widehat{G-e}$.
By the definition of $J$, we have $d_J(x')=d_G(x_{11})+d_G(x_{12})-4=d_G(x_{12})-1$ and
$d_J(y')=d_G(y_{21})+d_G(y_{22})-4=d_G(y_{22})-1$.
Recall  that $\delta(J)\ge3$.
This, together with $\Delta(G)=4$, implies that
$d_G(x_{12})=d_G(y_{22})=4$ and $d_J(x')=d_J(y')=3$.
Since each of $r$ and $s$ loses exactly one edge of $F$, we have $d_J(r)=d_J(s)=3$.
The three edges at each of $x_{11}$ and $y_{21}$ are either
bicontracted or deleted. Consequently, every edge of $J$ incident
with $x'$ or $y'$ is represented in $G$ by an edge incident with
$x_{12}$ or $y_{22}$, respectively.

Since $E_G(\{x_{11},x_{12}\},\{y_{21},y_{22}\})=\{x_{11}y_{21}\}$ and $F=\{x_{11}y_{21},x_{11}r,y_{21}s\}$, we have $x'y'\notin E(J)$.
Recall that $\widehat{G-e}$ is simple and $x_{11}r, y_{21}s\in E(G)$. Then $x_{12}r,y_{22}s\notin E(G)$, and so $x'r,y's\notin E(J)$.
Note that $x'$ and $s$ lie in the same colour class of $J$, and $y'$ and $r$ lie in the other colour class.

\smallskip
\textit{Subcase 2.1:} $J$ is isomorphic to $B_{10}$.
\smallskip
Use the hubs $a,b$ and rim vertices
$x_1,\ldots,x_4,y_1,\ldots,y_4$ of
Figure~\ref{fig:eight-cubic}.
The rim edges are $x_iy_i,x_iy_{i-1}$, with indices
modulo four, and the hub edges are $ay_i,bx_i$, $1\le i\le 4$.
By symmetry, we may assume $x'=x_1$.
The cubic nonneighbours of $x_1$ in the opposite colour class are $y_2,y_3$.
Since a symmetry of $B_{10}$ fixes $x_1$ and interchanges
$y_2,y_3$, we may assume that $y'=y_2$ and $r=y_3$.
The cubic nonneighbours of $y_2$ in the opposite colour class of $J$
are exactly $x_1,x_4$. Since $y_2s\notin E(J)$ and $s\ne x'$,
we have $s=x_4$.
Thus $x'=x_1$, $y'=y_2$, $e=uv$, and
$F=\{x_{11}y_{21},x_{11}y_3,x_4y_{21}\}$.
We illustrate the two graphs $J$ and $G$ in Figure \ref{fig:case2-b10}.

\begin{figure}[htbp]
\centering
\includegraphics[width=0.6\textwidth]{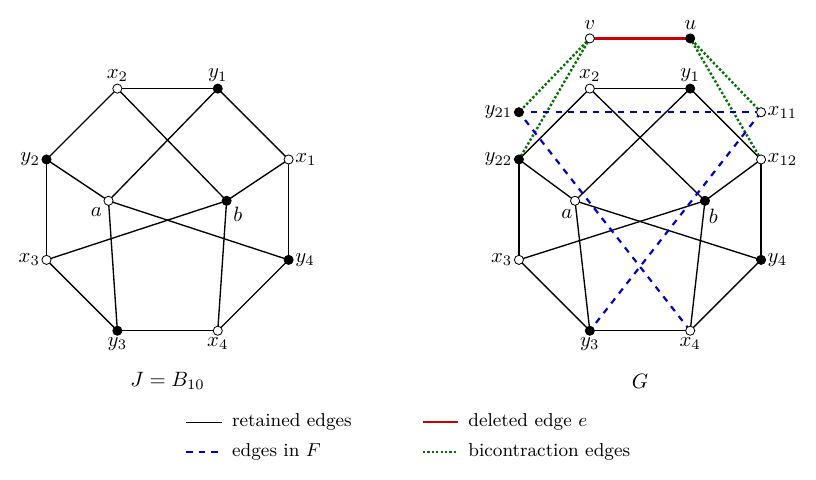}
\caption{Subcase 2.1 with $J=B_{10}$.}
\label{fig:case2-b10}
\end{figure}

Note that $x_4b\in E(G)$.
Suppose that  there exists a critical set $X\subseteq A_G$ for $x_4b$ in $G$.
By definition, we have $x_4\in X$, $N_G(b)\cap X=\{x_4\}$ and $|N_G(X)|=|X|+2$.
Since $N_G(b)=\{x_{12},x_2,x_3,x_4\}$, we have
$x_4\in X\subseteq\{x_4,a,v,x_{11}\}$ and so $|X|\le4$.
Note that
$N_G(x_4)=\{b,y_3,y_4,y_{21}\}$,
$N_G(a)=\{y_1,y_{22},y_3,y_4\}$,
$N_G(v)=\{u,y_{21},y_{22}\}$, and
$N_G(x_{11})=\{u,y_{21},y_3\}$.

If $X=\{x_4\}$, then $|N_G(X)|=4>|X|+2$,
a contradiction.
If $X=\{x_4,x_{11}\}$, then
$N_G(X)=\{b,y_3,y_4,y_{21},u\}$,
so $|N_G(X)|=5>|X|+2$, a contradiction.
If $a\in X$, then
$\{b,y_3,y_4,y_{21},y_1,y_{22}\}\subseteq N_G(X)$.
For $|X|\le3$, this gives $|N_G(X)|\ge6>|X|+2$,
a contradiction.
If $|X|=4$, then $X=\{x_4,a,v,x_{11}\}$.
Since $u\in N_G(v)$, we obtain
$|N_G(X)|=7>|X|+2$, again a contradiction.
The remaining sets are $\{x_4,v\}$ and
$\{x_4,v,x_{11}\}$.
Both sets have neighbourhood
$\{b,y_3,y_4,y_{21},u,y_{22}\}$ in $G$.
So $|N_G(X)|=6>|X|+2$, a contradiction.
Hence, there is no critical set contained in $A_G$ for $x_4b$ in $G$.
By Lemma~\ref{lem:critical-set}, $G-x_4b$ is a brace,
contradicting the minimality of $G$.

\smallskip
\textit{Subcase 2.2:} $J$ is isomorphic to $Q_{10}^{+}$.
\smallskip
Use the labelling of $Q_{10}^{+}$ in
Figure~\ref{fig:eight-cubic}, choosing the colour classes
so that $x'$ lies in the class containing $p_1,p_2,x_1,x_2,x_3$.
Recall that $d_J(x')=d_J(y')=d_J(r)=d_J(s)=3$, $x'y',x'r,y's\notin E(J)$,
$x'$ and $s$ lie in the same colour class of $J$, and $y'$ and $r$ lie in the other colour class.
Since $d_J(p_1)=4$, we have $x'\ne p_1$.
The vertex $p_2$ has only one cubic nonneighbour $q_2$ in the opposite
colour class, whereas $y',r$ are two distinct cubic nonneighbours
of $x'$. Thus $x'\ne p_2$.
Permuting the three pairs $x_i,y_i$, we may assume
$x'=x_1$, $y'=y_2$, and $r=y_3$.
The cubic nonneighbours of $y_2$ in the opposite colour class of $J$
are exactly $x_1,x_3$. Since $y's\notin E(J)$ and $s\ne x'$,
we have $s=x_3$.
Thus $x'=x_1$, $y'=y_2$, $e=uv$, and
$F=\{x_{11}y_{21},x_{11}y_3,x_3y_{21}\}$.
We illustrate the two graphs $J$ and $G$ in Figure \ref{fig:case2-q10}.
The vertices $p_1,q_1$ retain degree four in $G$, while $x_3,y_3$
each gain one edge of $F$ and also have degree four in $G$.
All four edges of the cycle $p_1q_1x_3y_3p_1$ in $J$ are retained
in $G$. Thus this cycle lies in $G[V_4]$, contradicting Lemma~\ref{lem:forest}.

\begin{figure}[htbp]
\centering
\includegraphics[width=0.6\textwidth]{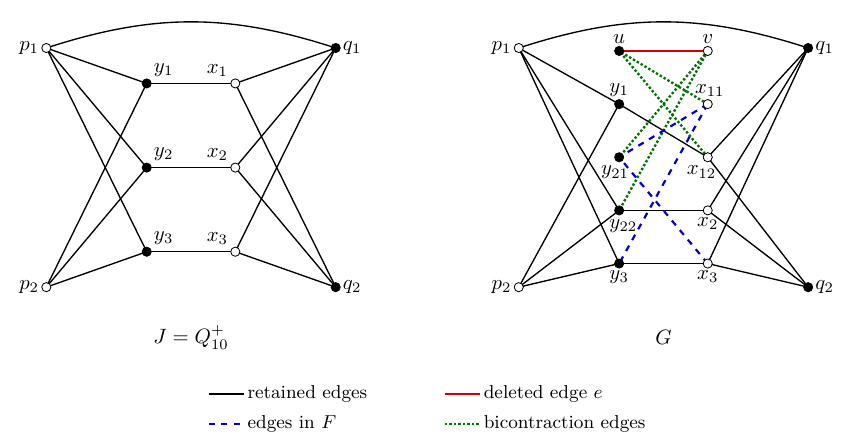}
\caption{Subcase 2.2 with $J=Q_{10}^{+}$.}
\label{fig:case2-q10}
\end{figure}

\medskip
\textbf{Case 3:} $k=1$ and $|F|=2$.
Then $|V(J)|=12$ and $|E(J)|=19$.
By  Lemma~\ref{lem:induction}\ref{item:induction-one-1},
$G$ is isomorphic to a stable-extension of $J$.
\smallskip

Let the attachment vertices be $a_1,a_2,b_1,b_2$, forming a stable
set in $J$, with $a_1,a_2\in A_J$ and $b_1,b_2\in B_J$.
The graph $G$ is obtained from $J$ by adding the two new vertices $x,y$ and
the edges $xb_1,xb_2,xy,ya_1,ya_2$, where
$x\in A_G$ and $y\in B_G$.
Then $\{a_1,a_2,b_1,b_2\}$ is a stable set in $G$.

Each colour class of $J$ has six vertices and degree sum $19$.
Since $\delta(J)\ge3$, its total degree excess above three is
$19-6\times3=1$. Thus $J$ has exactly one vertex of degree $4$  in
each colour class, say $p\in A_J$ and $q\in B_J$, and all other
vertices are cubic.
Neither $p$ nor $q$ is an attachment vertex, since each attachment
vertex gains one edge in $G$ and $\Delta(G)=4$.
Consequently, $V_4(G)=\{p,q,a_1,a_2,b_1,b_2\}$.

Note that $E(G[V_4])\subseteq\{pq,pb_1,pb_2,qa_1,qa_2\}$.
Lemma~\ref{lem:fourteen-critical-set} implies that $|E(G[V_4])|=|E(G[V_3])|$, and
Theorem~\ref{thm:transversal} implies that $|E(G[V_3])|\ge 3$.
By Lemma~\ref{lem:forest}, $G[V_4]$ is a forest and so $|E(G[V_4])|\le 5$.
Combining these, we can obtain $3\le|E(G[V_4])|=|E(G[V_3])|\le5$.
By Lemma~\ref{lem:fourteen-not}\ref{item:fourteen-not-three}, $|E(G[V_4])|=|E(G[V_3])|=3$ is impossible.
If $|E(G[V_4])|=5$, then $G[V_4]$ is a double star,
contradicting Lemma~\ref{lem:fourteen-not}\ref{item:fourteen-not-double}.
Hence $|E(G[V_4])|=4$.
By symmetry, we only need to consider the following two subcases:
$E(G[V_4])=\{pb_1,pb_2,qa_1,qa_2\}$ or $E(G[V_4])=\{pq,pb_1,pb_2,qa_1\}$.
For both subcases, write
$A_J=\{p,a_1,a_2,u_1,u_2,u_3\}$ and
$B_J=\{q,b_1,b_2,v_1,v_2,v_3\}$.
The vertices $u_1,u_2,u_3,v_1,v_2,v_3$ are cubic
in both $J$ and $G$.

\smallskip
\textit{Subcase 3.1: $E(G[V_4])=\{pb_1,pb_2,qa_1,qa_2\}$.}
Since $d_G(p)=4$ and $N_G(p)\setminus\{b_1,b_2\}\subseteq\{v_1,v_2,v_3\}$, we may assume that  $N_G(p)=\{b_1,b_2,v_1,v_2\}$.
For $i=1,2$, the vertex $b_i$ is adjacent  to
$p,x$ and exactly two of $u_1,u_2,u_3$ in $G$.
The vertices missed by $b_1,b_2$ must be distinct;
otherwise $|N_J(\{b_1,b_2\})|=3$, contradicting
Lemma~\ref{lem:brace}\ref{item:brace-neighbourhood}.
Relabelling $u_1,u_2,u_3$,  we may  assume that
$N_G(b_1)=\{p,x,u_2,u_3\}$ and
$N_G(b_2)=\{p,x,u_1,u_3\}$.
Recall that  each $a_i$ is adjacent to $q$ and exactly two
of $v_1,v_2,v_3$ in $J$. Then $|N_J(\{a_1,a_2\})|\le 4$.
By Lemma~\ref{lem:brace}\ref{item:brace-neighbourhood}, we have $|N_J(\{a_1,a_2\})|\ge 4$.
Consequently,
$N_J(\{a_1,a_2\})=\{q,v_1,v_2,v_3\}$ and
$N_G(\{a_1,a_2\})=\{q,y,v_1,v_2,v_3\}$.

Since $G$ is a minimal brace, $G-pb_i$ is not a brace, $i=1,2$.
By Lemma~\ref{lem:critical-set},  there exists a critical set $X\subseteq A_G$ for $pb_i$ in $G$.
By definition, we have $p\in X$, $N_G(b_i)\cap X=\{p\}$ and $|N_G(X)|=|X|+2$.
Since $N_G(b_i)=\{p,x,u_{3-i},u_3\}$, we have
$p\in X\subseteq\{p,a_1,a_2,u_i\}$ and so  $|X|\le4$.

If $X=\{p\}$, then
$|N_G(X)|=4>|X|+2$, a contradiction.
If $X\cap\{a_1,a_2\}\ne\emptyset$, then
$\{b_1,b_2,v_1,v_2,q,y\}\subseteq N_G(X)$.
For $|X|\le3$, this gives
$|N_G(X)|\ge6>|X|+2$, a contradiction.
If $|X|=4$, then $X=\{p,a_1,a_2,u_i\}$.
Since $N_G(\{a_1,a_2\})=\{q,y,v_1,v_2,v_3\}$, we obtain
$N_G(X)=\{b_1,b_2,q,y,v_1,v_2,v_3\}$.
Thus $|N_G(X)|=7>|X|+2$, again a contradiction.
Therefore $X=\{p,u_i\}$.
Since $|N_G(X)|=4=|N_G(p)|$, we have
$N_G(u_i)\subseteq N_G(p)$.
Since $b_i\notin N_G(u_i)$ and  $d_G(u_i)=3$, we have
$N_G(u_i)=N_G(p)\setminus\{b_i\}=\{b_{3-i},v_1,v_2\}$.
It follows that $N_G(u_1)=\{b_2,v_1,v_2\}$ and $N_G(u_2)=\{b_1,v_1,v_2\}$.
Combining these with $N_G(p)=\{b_1,b_2,v_1,v_2\}$, we obtain
$N_G(\{p,u_1,u_2\})=\{b_1,b_2,v_1,v_2\}$, contradicting Lemma~\ref{lem:brace}\ref{item:brace-neighbourhood} since $G$ is a brace.

\smallskip
\textit{Subcase 3.2: $E(G[V_4])=\{pq,pb_1,pb_2,qa_1\}$.}
\smallskip
Since $d_J(a_2)=3$ and $N_J(a_2)\cap \{b_1,b_2,q\}=\emptyset$, we have $N_J(a_2)=\{v_1,v_2,v_3\}$.
Since $d_J(p)=4$, we may relabel $v_1,v_2,v_3$ so that
$N_J(p)=\{q,b_1,b_2,v_3\}$.

Since  $J$ is a minimal brace,  $J-pq$ is not a brace. By Lemma~\ref{lem:critical-set}, there is a critical set
$X\subseteq A_J$ for $pq$ in $J$.
It follows that  $1\le|X|\le |A_J|-2=4$,  $p\in X$, $N_J(q)\cap X=\{p\}$, and $|N_J(X)|=|X|+2$.
Since $d_J(p)=4$, we have $|X|\ne1$.
Moreover, $X$ excludes the three vertices of
$N_J(q)\setminus\{p\}$, and so $|X|\le3$.
Thus $|X|\in\{2,3\}$.
If $|X|=3$, Lemma~\ref{lem:critical-set} implies that there is a  dual critical set of $X$ with  size $|A_J|-|X|-1=2$ in $J$. Hence, there exists a critical set with size two in $A_J$ or $B_J$ for $pq$ in $J$.

For a critical set $\{p,w\}\subseteq A_J$ for $pq$ in $J$, we have
$|N_J(\{p,w\})|=4=|N_J(p)|$ and $q\notin N_J(w)$.
Therefore $N_J(w)\subseteq N_J(p)\setminus\{q\}$.
Since $\delta(J)\ge3$, we have $N_J(w)=N_J(p)\setminus\{q\}$.
Likewise, a critical set $\{q,w'\}\subseteq B_J$ for $pq$ in $J$
would give $N_J(w')=N_J(q)\setminus\{p\}$.
If $w'\in\{v_1,v_2,v_3\}$, then $a_2\in N_J(w')$
but $a_2\notin N_J(q)$, a contradiction. If  $w'\in\{b_1,b_2\}$, then $a_1\notin N_J(w')$
but $a_1\in N_J(q)\setminus\{p\}$, a contradiction.
Thus, there is no critical set of size $2$ contained in $B_J$ for $pq$ in $J$.
It follows that there exists a critical set $\{p,w\}\subseteq A_J$ for $pq$ in $J$.
Moreover, $N_J(w)=N_J(p)\setminus\{q\}=\{b_1,b_2,v_3\}$.
Since neither $a_1$ nor $a_2$ is adjacent to $b_1,b_2$ in $J$,
we have $w\in\{u_1,u_2,u_3\}$.
Relabelling these vertices, assume that $w=u_3$.
Then $N_J(u_3)=\{b_1,b_2,v_3\}$.

Note that $d_J(q)=4$ and $N_J(q)\setminus\{a_1,p\}\subseteq\{u_1,u_2,u_3\}$.
Since $qu_3\notin E(J)$, we have
$N_J(q)=\{p,a_1,u_1,u_2\}$.
Note that $N_J(v_3)=\{p,a_2,u_3\}$ as $v_3$ is cubic in $J$.
Then $a_1v_3\notin E(J)$.
Since $qa_1\in E(J)$, $d_J(a_1)=3$, and
$N_J(a_1)\setminus\{q\}\subseteq\{v_1,v_2\}$, we obtain
$N_J(a_1)=\{q,v_1,v_2\}$.

In $G$, we have $N_G(a_1)=\{q,y,v_1,v_2\}$,
$N_G(a_2)=\{y,v_1,v_2,v_3\}$, $N_G(x)=\{y,b_1,b_2\}$,
$N_G(u_3)=\{b_1,b_2,v_3\}$, and $N_G(q)=\{p,a_1,u_1,u_2\}$.
Since $G$ is a minimal brace, $G-a_1q$ is not a brace. By Lemma~\ref{lem:critical-set},  there exists a critical set $X\subseteq A_G$ for $a_1q$ satisfying
$a_1\in X\subseteq\{a_1,a_2,x,u_3\}$
and $|N_G(X)|=|X|+2$ in $G$.
We next show that the latter equality never holds, a contradiction.

If $X=\{a_1\}$, then $|N_G(X)|=4>|X|+2$.
Note that each of $a_2$, $x$,  and $u_3$ has a neighbour outside $N_G(a_1)$.
If $|X|=2$, then $|N_G(X)|\ge5>|X|+2$.
If $|X|\ge3$, then $X$ contains at least two of $a_2$, $x$,  and $u_3$.
It follows that $N_G(X)=\{q,y,b_1,b_2,v_1,v_2,v_3\}=B_G$.
Thus $|N_G(X)|=7>|X|+2$ because $|X|\le4$.
In every case the required equality fails.
\end{proof}

\begin{proof}[Proof of Theorem~\ref{thm:main}\ref{item:eight}]

Suppose that $G$ is isomorphic to one of $B_8, B_{10}, Q_{10}^{+}$ and $Q_{12}$.
Then $G$ is a minimal brace by Lemma~\ref{lem:small},
and has exactly eight cubic vertices by Proposition~\ref{prop:parameters}.

Conversely, suppose that $G$ is a minimal brace of order $n\ge 6$ which has exactly eight cubic vertices.
By Lemma~\ref{lem:eight-reduction}, either $G$ is isomorphic to one of $B_8, B_{10}, Q_{10}^{+}$ and $Q_{12}$, or $n\in\{14,16\}$.
By Lemmas~\ref{lem:sixteen} and \ref{lem:fourteen}, we have $n\ne14$ and $n\ne16$, which implies that $G$ is isomorphic to one of $B_8, B_{10}, Q_{10}^{+}$ and $Q_{12}$.
This completes the proof of Theorem \ref{thm:main}\ref{item:eight}.
\end{proof}

\begin{proposition}\label{prop:four-lines}
Let $G$ be a minimal brace of order at least $6$ with at most eight
cubic vertices. Then $G$ has at least four $3$-lines, with equality
if and only if $G$ is isomorphic to $ Q_{12}$.
\end{proposition}

\begin{proof}
By Theorem~\ref{thm:main}, $G$ is isomorphic to one of $K_{3,3}, B_8, B_{10}, Q_{10}^{+}$ and $Q_{12}$.
Note that $K_{3,3}$ has nine $3$-lines. By Proposition~\ref{prop:parameters}, each of $B_8, B_{10}, Q_{10}^{+}$ and $Q_{12}$ has at least four $3$-lines, with equality if and only if $G$ is isomorphic to $ Q_{12}$. The result holds.
\end{proof}

\section{Concluding remarks}\label{sec:conclusion}
By Theorem~\ref{thm:transversal}, for any minimal brace $G$
of order at least $6$,   every perfect
matching of $G$ contains a $3$-line and hence $G$ has at least three $3$-lines.
Proposition~\ref{prop:four-lines} shows that the lower bound of
three $3$-lines is not attained by any graph with at most eight
cubic vertices. Then the following problem is of interest.

\begin{question}\label{question:four-lines}
Does every minimal brace of order at least six contain at least four
$3$-lines?
\end{question}

An affirmative answer would be best possible, since $Q_{12}$ has
exactly four $3$-lines by Proposition~\ref{prop:parameters}.
A counterexample would have exactly three $3$-lines by
Theorem~\ref{thm:transversal}, at least nine cubic vertices by
Proposition~\ref{prop:four-lines}, and its $3$-lines would form
either a star or a matching by Lemma~\ref{lem:three-shapes}.

\section*{Acknowledgements}
This work is supported by the National Natural Science Foundation of China (Nos. 12501480, 12571381 and 12371361) and Natural Science Foundation of Henan Province (No. 252300421786).

\end{document}